\documentclass[12pt,twoside]{amsart}
\usepackage{amsmath}
\usepackage{amssymb}
\usepackage{amscd}
\usepackage{xypic}
\xyoption{all}
\title[A note on the ring of invariant jet differentials]{A note on the ring of invariant jet differentials}

\author{Mohammad Reza Rahmati}
\thanks{}
\address{ Abdus Salam school of Mathematical Sciences, GCU, Lahore, Pakistan
\hfill\break 
\hfill\break \\
\hfill\break }
\email{mrahmati@cimat.mx, rewayat.khan@gmail.com}

\newcommand{\comments}[1]{}

\newtheorem{theorem}{Theorem}[section]

\newtheorem{remark}[theorem]{Remark}

\keywords{Demailly-Semple tower, Jets of entire curves, Differential Galois group, Differential field, Wronskian, Ring of invariant jets}


\begin{document}

\begin{abstract}
In this article we briefly discuss the finite generation of fiber rings of invariant $k$-jets of holomorphic curves in a complex projective manifold, using differential Galois theory. 
\end{abstract}

\maketitle



\section{Introduction}

\vspace{0.3cm}

Let $X$ be a $n$-dimensional complex projective manifold. In \cite{GG} Green and Griffiths introduce the bundle of germs of $k$-jet differentials on the manifold $X$ as a sequence of projective bundles  $\pi_k:X_k \to X_{k-1} \to ...\to X_1 \to X$ which are defined inductively as in \cite{D2}. The fibers of $X_k$ at $x \in X$ is the set of equivalence classes of germs of holomorphic maps $f:(\mathbb{C},0) \to (X,x)$ with equivalence relation 

\begin{equation}
f \equiv_k g \ \ \ \Leftrightarrow \ \ \ f^{(j)}(0)=g^{(j)}(0), \ (0 \leq j \leq k).
\end{equation}

\noindent
By choosing local holomorphic coordinates around $x$, the elements of the fiber $J_{k,x}$ can be represented by the Taylor expansion 

\begin{equation}
f(t)=tf'(0)+\frac{t^2}{2!}f''(0)+...+\frac{t^k}{k!}f^{(k)}(0)+O(t^{k+1}).
\end{equation}

\noindent
Setting $f=(f_1,...,f_n)$ on open neighborhoods of $0 \in \mathbb{C}$, the fiber is 

\begin{equation}
J_{k,x}=\{(f'(0),...,f^{(k)}(0))\} = \mathbb{C}^{nk}.
\end{equation}

\noindent
The action of $\mathbb{C}^*$ on $k$-jets is $\lambda. (f'(0),...,f^{(k)}(0)) =(\lambda .f'(0),...,\lambda^k .f^{(k)}(0))$. The Green-Griffiths bundles $E_{k,m}T_X^*$ are defined by $J_k/\mathbb{C}^*$. Alternatively one defines sheaf of sections $E_{k,m}V^*$ of weighted homogeneous polynomials along the fibers in $X_k$ in the jet coordinates $\xi_1,...,\xi_k$ with weights $(1,2,...,k)$ respectively and set $E_{k,m}V^*=\bigoplus_m E_{k,m}V^*$, see \cite{D2}. 

Let $G_k$ be the group of local holomorphic automorphisms of $(\mathbb{C},0)$ of the form 

\begin{equation}
t \longmapsto \phi(t)=\alpha_1t+\alpha_2t^2+...+\alpha_kt^k, \qquad \alpha_1 \in \mathbb{C}^*. 
\end{equation}

\noindent
Its action on the $k$-jets is given by the following matrix multiplication

\begin{equation}
[f'(0), f''(0)/2!, ..., f^{(k)}(0)/k!].
\left[ \begin{array}{ccccc}
a_1 & a_2 & a_3 & ... & a_k \\
0 & a_1^2 & 2a_1a_2 & ... & a_1a_{k-1}+...a_{k-1}a_1\\
0 & 0 & a_1^3 & ... & 3a_1^2a_{k-2}+...\\
. & . & . & ... & .\\
0 & 0 & 0 & ... & a_1^k 
\end{array} \right ].
\end{equation}

\noindent
The unipotent elements $U_k$ correspond to the matrices with $a_1=1$. We have $G_k=\mathbb{C}^* \times U_k$. Recall a $k$-jet $f:\mathbb{C} \to X$ is regular if $f'(0) \ne 0$. There are  embeddings 

\begin{equation}
\Phi:J_k^{reg}/G_k \hookrightarrow Grass(k, \bigoplus_{l\leq k}Sym^{l} \mathbb{C}^n) \hookrightarrow \mathbb{P}(\bigwedge^k \bigoplus_{l\leq k}Sym^{l} \mathbb{C}^n), 
\end{equation}

\noindent
where the second embedding is the Plucker embedding. Let $e_1 ,..,e_n$ be standard basis of $\mathbb{C}^n$. Then a basis of $\bigoplus_{l\leq k}Sym^{l} \mathbb{C}^n$ consists of 
$\{e_{i_1...i_s}=e_{i_1}...e_{i_s}, s \leq k\}$. Then a basis of $\mathbb{P}(\wedge^k \bigoplus_{l\leq k}Sym^{l} \mathbb{C}^n)$ is $\{e_{i_1}\wedge ... \wedge e_{i_s}, s \leq k $. Write

\begin{equation}
z=\Phi(e_1,...,e_k)=[e_1 \wedge (e_2 +e_1^2) \wedge ... \wedge ( \sum e_{i_1}...e_{i_k})].
\end{equation}

\noindent
The group $Gl(n)$ acts on $\Phi(e_1,...,e_k)$ by acting on each $e_i$ in the expression (7). The Plucker embedding above uses $SL_k$-orbit of the point $z$ as a highest weight vector of some representation. The properties of these orbits and also their compactifications has been studied in \cite{BK}. The group $G_k$ of transformations presented in (5) are not reductive, and their invariant theory do not satisfy the properties of the invariants of reductive groups. Specifically the problem of finite generation of the invariant fiber rings for jet differentials is still an open problem. 

\vspace{0.3cm}

\section{Formalism of Differential fields}

\vspace{0.3cm}

The rings we are considering in this section are all commutative with 1. A derivation of the ring $A$ is an additive mapping 

\begin{equation}
\delta:A \to A, \qquad \delta(ab)=b\delta(a) + a \delta(b).
\end{equation}

\noindent
Let $A$ be a differential field and $B$ a differential subfield. The differential Galois group $G$ of $A/B$ is the group of all differential automorphisms of $A$ living $B$ fixed. Then the same formalism like the Galois groups of usual fields appear here also. For any intermediate differential subfield $C$, denote the subgroup of $G$ living $C$ elementwise fixed by $C'$; and similar for any subgroup $H$ of $G$ denote by $H'$ the elements in $A$ fixed by that. Call a field or group closed if it is equal to its double prime. Now with these notations making PRIMED defines the Galois correspondence between closed subgroups and closed differential subfields.

The Wronskian of $n$ elements $y_1,...,y_n$ in a differential ring is defined as the determinant 

\begin{equation}
W(y_1,...,y_n)=\left|
\begin{array}{cccc}
y_1 & y_2 & ... & y_n \\
y_1' & y_2' & ... &y_n' \\
   &     &     &    \\
y_1^{(n)} &     &   &y_n^{(n-1)} 
\end{array} 
\right|.   
\end{equation}

\noindent
It is a quite well known that, $n$ elements in a differential field are linearly dependent over the field of constants if and only if their Wronskian vanishes. We will call an extension of the form $A=K\langle u_1,...,u_n\rangle$ with $u_1,...,u_n$ are solutions of 

\begin{equation}
L(y)=\frac{W(y,u_1,...,u_n)}{W(u_1,...,u_n)}=y^{(n)}+a_1y^{(n-1)}+...+a_ny=0
\end{equation}

\noindent
a Picard extension, cf. \cite{K}. 

\begin{theorem} \cite{K} We have the following
\begin{itemize}
\item[(1)] Let $K \subset L \subset M$ be differential fields. Suppose that $L$ is Picard over $K$ and $M$ has the same field of constants as $K$. Then any differential automorphism of $M$ over $K$ sends $L$ into itself.
\item[(2)] The differential Galois group of a Picard extension is an algebraic matrix group over the field of constants. 
\item[(3)] If $K$ has an algebraically closed constant field of characteristic $0$, and $M$ a Picard extenstion of $K$, then any differential isomorphism over $K$ between two intermediate fields extends to the whole $M$. In particular this also holds for any differential automorphism of an intermediate field 
over $K$. 
\item[(4)] Galois theory implements a one-to-one correspondence between the intermediate differential fields and the algebraic subgroups of the differential Galois group $G$. A closed subgroup $H$ is normal iff the corresponding field $L/K$ is normal, then $G/H$ is the full differential Galois group of $L$ over $K$.
\end{itemize} 
\end{theorem}

\noindent
In fact over a constant field of $char=0$ any differential isomorphism between intermediate fields extends to the whole differential field.
Let $A=K\langle u_1,...,u_n\rangle$ be a Picard extension and $W$ the Wronskian of $u_1,...,u_n$. A basic fact about the Wronskians is that for a differential automorphism $\sigma$ of $A$; we have $\sigma(W)=|c_{ij}|W$. Therefore $W$ is fixed by $\sigma$ if and only if $ |c_{ij}| =1$. 

A family of elements $(x_i)_{i \in I}$ is called differential algebraic independent; if the family $(x_i^{(j)})_{i \in I,j \geq 0}$ is algebraically independent over the field of constants, otherwise we call them dependent. An element $x$ is called differentially algebraic if the family consisting of $x$ only, is differential algebraic dependent. An extension is called differential algebraic if any element of it, is so. Finally we say $G$ is differentially finite generated over $F$ if there exists elements $x_1,...,x_n \in G$ such that $G$ is generated over $F$ by the family $(x_i^{(j)})_{1 \leq i \leq n,j \geq 0}$, cf. \cite{K}.

\begin{theorem} \cite{K}
Let $F \subset G$ be an extension of differential fields, then

\begin{itemize}
\item If $G=F\langle x_1,...,x_n\rangle$ and each $x_i$ is differential algebraic over $F$ then $G$ is finitely generated over $F$.
\item If $G$ is differential finite generated over $F$ and $F \subset E \subset G$ is an intermediate differential field, the $E$ is also differentially finite generated.
\end{itemize}
\end{theorem}

\begin{remark} \cite{K}
Let $J=\mathbb{C}\langle \xi_1,,,\xi_k \rangle$ be a differential field obtained by adjoining $n$-differential indeterminates. Assume $g$ is a linear transformation that is given by the matrix $(c_{ij})$ on the variables $\xi_i$. Define $g$ on all the differential variables by

\begin{equation}
g.\ \xi_i^{(m)}=\sum c_{ij}\xi_j^{(m)}, \qquad m \geq 0.
\end{equation}

\noindent
Then $g$ is a differential automorphism of $J$. Define $L(y)$ by (11). Then $L(y)=0$ is a linear differential equation which $\xi_i$ are its indepentent solutions. $J$ is a Picard extension of $\mathbb{C}$ and its differential Galois group is the full linear group.
\end{remark}

\vspace{0.5cm}

\noindent
\textbf{Applications:} Apply the Galois theory to the coordinate ring of fibers of jet bundle. The fiber rings of the Green-Griffiths bundles $X_k$ and sheaves $E_{k,m}V^*$ are differential rings. We shall consider their quotient fields. The algebraic groups $GL_k=\mathbb{C}^* \times U_k$, $SL_k$ and also $G_k$ act linearly on differential variables and are differential Galois groups. As we explained the fixed field of $SL_k$ and the Galois group $1$ are the fields generated by the Wronskians and the Whole quotient field of the fibers. Therefore the middle group $U_k$ also has finitely generated 
fixed field, where we have used the criteria in the Theorems 2.1 and 2.2. Furthermore one finds that a possible choice of generators may include the fixed generators of $SL_k$, i.e. the Wronskians. By the Noether normalization theorem, there exists a finite number of generators $\wp_1,...,\wp_l$ such that the ring of fibers in $J_k(X)^{G_k}$ is algebraic over $\mathbb{C}\langle \wp_1,...,\wp_l\rangle$. It follows that $\mathbb{C}[(J_{k,x}(X)]^{G_k}=\mathbb{C}\langle \wp_1,...,\wp_l \rangle (\alpha_1,...,\alpha_n)$.

\begin{remark} \cite{M}
Assume $P$ and $Q$ are two local sections of the Green-Griffiths bundle, then the the operator $\nabla_j:f \mapsto f_j'$ is invariant under change of parameter on $\mathbb{C}$. Define a bracket operation as follows 

\begin{equation}
[P,Q]=\frac{1}{deg(P)}PdQ-\frac{1}{deg(Q)}QdP.
\end{equation}

\noindent
Later one successively defines the brackets

\begin{equation}
\begin{split}
[\nabla_j,\nabla_k] &=f_j'f_k''-f_j''f_k'\\
[\nabla_j,[\nabla_k, \nabla_l]] & =f_j'(f_k'f_l''-f_l'f_k'')-3f_j'(f_k'f_l''-...).
\end{split}
\end{equation}

\noindent
The sections produced by brackets generate the fiber rings of invariant jet section. 
\end{remark}

\noindent
ACKNOWLEDGEMENT: I thank the Abdus Salam School of Mathematical Sciences , GCU, for its research facilities and financial support.


\begin{thebibliography}{99}

\bibitem{BK} G. Berczi, F. Kirwan, A geometric construction for invariant jet differentials, preprint    

\bibitem{D2}  J. P. Demailly, Hyperbolic algebraic varieties and holomorphic differential equations, Acta Math. Vietnam, 37(4) 441-512, 2012

\bibitem{GG} M. Green, P. Griffiths, Two applications of algebraic geometry to entire holomorphic mappings, The Chern Symposium 1979. (Proc. Intern. Sympos., Berkeley, California, 1979) 41-74, Springer, New York, 1980.

\bibitem{K} I Kaplansky, An introduction to differential algebra, Pub. de L'institut de Math. de Univ. Nancao, Hermann Paris 1957

\bibitem{M} J. Merker, Application of computational invariant theory to Kobayashi hyperbolicityand to Green-Griffiths algebraic degeneracy, Journal of Symbolic Computation 45, 986-1074 (2010)

\end{thebibliography}
\end{document}